%% file: lwqo.tex
\documentclass[10pt]{article}

\usepackage{lmodern}
\usepackage[T1]{fontenc}
\usepackage{amsmath}
\usepackage{amsthm}
\usepackage{amssymb}
\usepackage{mathabx} 
\usepackage{url}
\usepackage{latexsym}
\usepackage{titlefoot}
\usepackage[small]{titlesec}
\usepackage{units} 
\usepackage[small,it]{caption}

\usepackage[square,comma,numbers,sort&compress]{natbib}

\usepackage{tikz}
\usetikzlibrary{arrows, automata, decorations.pathreplacing, fit, matrix, patterns, positioning}
\usepackage{xifthen} 

\usepackage{color}
\definecolor{lightgray}{rgb}{0.8, 0.8, 0.8}
\definecolor{darkgray}{rgb}{0.7, 0.7, 0.7}

\usepackage[bookmarks]{hyperref}
\hypersetup{
	colorlinks=true,
	linkcolor=black,
	anchorcolor=black,
	citecolor=black,
	urlcolor=black,
	pdfpagemode=UseThumbs,
	pdftitle={A counterexample regarding labelled well-quasi-ordering},
	pdfsubject={Combinatorics},
	pdfauthor={Robert Brignall, Michael Engen, and Vincent Vatter},
}

\newcounter{todocounter}

\newcommand{\minisec}[1]{\noindent{\sc #1.}}

\theoremstyle{plain}
\newtheorem{theorem}{Theorem}[section]
\newtheorem{proposition}[theorem]{Proposition}
\newtheorem{corollary}[theorem]{Corollary}
\newtheorem{conjecture}[theorem]{Conjecture}

\newtheorem*{higmans-lemma*}{Higman's Lemma~\cite{higman:ordering-by-div:}}

\makeatletter
\newfont{\footsc}{cmcsc10 at 8truept}
\newfont{\footbf}{cmbx10 at 8truept}
\newfont{\footrm}{cmr10 at 10truept}
\renewenvironment{abstract}{
	\begin{list}{}%
	{\setlength{\rightmargin}{1in}%
	\setlength{\leftmargin}{1in}}%
	\item[]\ignorespaces\begin{small}}%
	{\end{small}\unskip\end{list}%
}

\newcommand{\C}{\mathcal{C}}

\newcommand{\st}{\::\:}
\input{pp-doodles}

\datefoot{\today}

\newpagestyle{main}[\small]{
	\headrule
	\sethead[\usepage][][]
	{\sc A Counterexample Regarding Labelled Well-Quasi-Ordering}{}{\usepage}
}

\title{\sc A Counterexample Regarding Labelled Well-Quasi-Ordering}

\author{\centering
\begin{tabular}{ccc}
Robert Brignall
&\rule{0pt}{0pt}&
Michael Engen\quad and\quad Vincent Vatter\footnote{Vatter's research was sponsored by the National Security Agency under Grant Number H98230-16-1-0324. The United States Government is authorized to reproduce and distribute reprints not-withstanding any copyright notation herein.}\\[-0.25ex]
\small School of Mathematics and Statistics
&&
\small Department of Mathematics\\[-0.5ex]
\small The Open University
&&
\small University of Florida\\[-0.5ex]
\small Milton Keynes, England UK
&&
\small Gainesville, Florida USA\\[-1.5ex]
\end{tabular}
}

\titleformat{\section}{\large\sc}{\thesection.}{1em}{}
\date{}
 
\begin{document}
\maketitle

\pagestyle{main}

\begin{abstract}
Korpelainen, Lozin, and Razgon conjectured that a hereditary property of graphs which is well-quasi-ordered by the induced subgraph order and defined by only finitely many minimal forbidden induced subgraphs is labelled well-quasi-ordered, a notion stronger than that of $n$-well-quasi-order introduced by Pouzet in the 1970s. We present a counterexample to this conjecture. In fact, we exhibit a hereditary property of graphs which is well-quasi-ordered by the induced subgraph order and defined by finitely many minimal forbidden induced subgraphs yet is not $2$-well-quasi-ordered. This counterexample is based on the widdershins spiral, which has received some study in the area of permutation patterns.
\end{abstract}

\section{Well-Quasi-Order and Strengthenings}

A \emph{hereditary property}, or (throughout this paper) \emph{class} of graphs is a set of finite graphs that is closed under isomorphism and closed downward under the induced subgraph ordering. Thus if $\C$ is a class, $G\in\C$, and $H$ is an induced subgraph of $G$, then $H\in\C$. Many natural sets of graphs form classes, such as the perfect graphs or permutation graphs. For an extensive survey we refer to the encyclopedic text of Brandst\"adt, Le, and Spinrad~\cite{brandstadt:graph-classes:-:}. A common way to describe a graph class is via its set of \emph{minimal forbidden induced subgraphs}, that is, the minimal (under the induced subgraph order) graphs which do not lie in the class. The set of minimal forbidden induced subgraphs necessarily forms an \emph{antichain}, meaning that none of them is contained in another.

We are interested here in well-quasi-ordering. A \emph{quasi-order} (a reflexive and transitive but not-necessarily-irreflexive binary relation) is a \emph{well-quasi-order} (\emph{wqo}), or is \emph{well-quasi-ordered} (also \emph{wqo}), if it does not contain an infinite antichain or an infinite strictly decreasing chain. As the set of finite graphs under the induced subgraph relation cannot contain an infinite strictly decreasing chain, in this case well-quasi-order is synonymous with the absence of infinite antichains.

Well-quasi-ordering has been studied for a wide variety of combinatorial objects, under many different orders; see Huczynska and Ru\v{s}kuc~\cite{huczynska:well-quasi-orde:} for a recent survey. Thus while the celebrated Minor Theorem of Robertson and Seymour~\cite{robertson:graph-minors-i-xx:} shows that the family of all graphs is wqo under the minor order, one might instead ask about the induced subgraph order, and in this context the set of graphs is clearly not wqo. For example, the set of chordless cycles forms an infinite antichain. Another infinite antichain is the set of \emph{double-ended forks}, examples of which are shown in Figure~\ref{fig-double-ended-forks}.

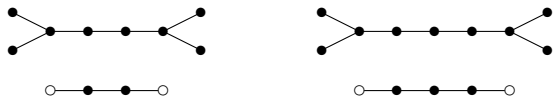
\begin{figure}
\begin{center}
	\begin{tabular}{ccc}
	\begin{tikzpicture}[scale=0.5]
		\plotpartialperm{-1/-0.5,-1/0.5,0/0,1/0,2/0,3/0,4/-0.5,4/0.5};
		\draw (-1,-0.5)--(0,0)--(-1,0.5);
		\draw (4,-0.5)--(3,0)--(4,0.5);
		\draw (0,0)--(3,0);
	\end{tikzpicture}
	&\quad\quad&
	\begin{tikzpicture}[scale=0.5]
		\plotpartialperm{-1/-0.5,-1/0.5,0/0,1/0,2/0,3/0,4/0,5/-0.5,5/0.5};
		\draw (-1,-0.5)--(0,0)--(-1,0.5);
		\draw (5,-0.5)--(4,0)--(5,0.5);
		\draw (0,0)--(4,0);
	\end{tikzpicture}
	\\
	\begin{tikzpicture}[scale=0.5]
		\draw (0,0)--(3,0);
		\absdothollow{(0,0)};
		\absdothollow{(3,0)};
		\plotpartialperm{1/0,2/0};
	\end{tikzpicture}
	&\quad\quad&
	\begin{tikzpicture}[scale=0.5]
		\draw (0,0)--(4,0);
		\absdothollow{(0,0)};
		\absdothollow{(4,0)};
		\plotpartialperm{1/0,2/0,3/0};
	\end{tikzpicture}
\end{tabular}
\end{center}
\caption{The double-ended forks form an infinite antichain in the induced subgraph order; two members of this family are shown on the top row. The bottom row demonstrates that the set of paths labelled by a two-element antichain form an infinite antichain in the labelled induced subgraph order (instead of labelling the vertices, in this figure we have colored them black and white).}
\label{fig-double-ended-forks}
\end{figure}

However, many important graph classes are nonetheless known to be wqo. For example, the class of \emph{co-graphs} (those which do not contain an induced copy of $P_4$) is wqo~\cite{damaschke:induced-subgrap:}. In fact, this class satisfies an even stronger property. Let $(X,\le)$ be any quasi-order. A \emph{labelling} of the graph $G$ is a function $\ell$ from the vertices of $G$ to $X$, and the pair $(G,\ell)$ is called a \emph{labelled graph}. A labelled graph $(H,k)$ is a \emph{labelled induced subgraph} of $(G,\ell)$ if $H$ is isomorphic to an induced subgraph of $G$ and the isomorphism maps each vertex $v\in H$ to a vertex $w\in G$ such that $k(v)\le \ell(w)$ in $(X,\le)$. The class $\C$ of (unlabelled) graphs is \emph{labelled well-quasi-ordered (lwqo)} if for every wqo $(X,\le)$ the set of all graphs in $\C$ labelled by $(X,\le)$ is wqo by the labelled induced subgraph order. This is equivalent to saying that the set of graphs in $\C$ labelled by $(X,\le)$ does not contain an infinite antichain. It follows from the work of Atminas and Lozin~\cite{atminas:labelled-induce:} that the class of co-graphs is lwqo.

The lwqo property is much stronger than wqo. For one example, the set of all chordless paths, which is trivially wqo, is not lwqo, as indicated in Figure~\ref{fig-double-ended-forks}. This shows that the class of \emph{linear forests}, that is, the class of all graphs whose connected components are paths, is not lwqo, despite being wqo. As further evidence of the strength of lwqo, we present the following result and its short proof.

\begin{proposition}[Pouzet~\cite{pouzet:un-bel-ordre-da:}; see also Daligault, Rao, Thomass{\'e}~\cite{daligault:well-quasi-orde:}]
\label{prop-lwqo-implies-finite-basis}
Every lwqo class of graphs is defined by finitely many minimal forbidden induced subgraphs.
\end{proposition}
\begin{proof}
Suppose that the graph class $\C$ is lwqo but is defined by an infinite set, say $B$, of minimal forbidden induced subgraphs. We take our wqo set of labels to be the $2$-element antichain consisting of the colors black and white. For each graph $G\in B$ choose an arbitrary vertex $v$ of $G$, so $G-v\in\C$. We now label/color the vertices of $G-v$ white if they are adjacent to $v$ in $G$ and black otherwise to obtain a labelled graph $G'$. Let $B'$ denote the set of all labelled graphs formed in this manner. As $B$ is an infinite antichain, it follows that $B'$ is an infinite antichain of labelled graphs from $\C$, but this contradicts our hypothesis that $\C$ is lwqo.
\end{proof}
Thus an lwqo graph class is necessarily wqo and defined by finitely many minimal forbidden induced subgraphs. Korpelainen, Lozin, and Razgon~\cite{korpelainen:boundary-proper:} conjectured that the converse also holds, a conjecture that was later repeated by Atminas and Lozin~\cite{atminas:labelled-induce:} in their work on lwqo.

\begin{conjecture}[Korpelainen, Lozin, and Razgon~\cite{korpelainen:boundary-proper:}]
\label{false-conj-lwqo}
A class of graphs which is wqo by the induced subgraph relation is lwqo if and only if it is defined by finitely many minimal forbidden induced subgraphs.
\end{conjecture}

We disprove Conjecture~\ref{false-conj-lwqo}. In particular, we construct a class of (permutation) graphs which is wqo and defined by finitely many minimal forbidden induced subgraphs, but is not lwqo. Note that the class of linear forests does not present a counterexample to Conjecture~\ref{false-conj-lwqo}, as the set of minimal forbidden subgraphs of this family consists of $K_{1,3}$ and the chordless cycles, $\{C_n\st n\ge 3\}$.

Notions of order between wqo and lwqo have been studied as far back as 1972, by Pouzet~\cite{pouzet:un-bel-ordre-da:}. He defined the class $\C$ to be \emph{$n$-well-quasi-ordered} (\emph{$n$-wqo}) if the set of all graphs in $\C$ labelled by an $n$-element antichain does not contain an infinite antichain in the labelled induced subgraph order. Pouzet stated the following tantalizing conjecture, which remains open.

\begin{conjecture}[Pouzet~\cite{pouzet:un-bel-ordre-da:}]
A class of graphs is $2$-wqo if and only if it is $n$-wqo for every $n\ge 2$.
\end{conjecture}

Obviously an lwqo class is also $n$-wqo for all $n\ge 2$ and is in particular $2$-wqo. Moreover, our proof of Proposition~\ref{prop-lwqo-implies-finite-basis} only used a $2$-element antichain\footnote{For this reason, Proposition~\ref{prop-lwqo-implies-finite-basis} can be strengthened to state that every $2$-wqo class of graphs is defined by finitely many minimal forbidden induced subgraphs.}, and the counterexample we present to Conjecture~\ref{false-conj-lwqo} also fails to be $2$-wqo. Perhaps something far stronger is true: is a class of graphs lwqo if and only if it is $2$-wqo?

\section{Permutations and Permutation Graphs}
\label{sec-perm-graphs}

Our counterexample to Conjecture~\ref{false-conj-lwqo} consists of a class of permutation graphs, and its construction makes use of several tools which have been introduced in the study of permutation patterns. For a broad overview of permutation patterns, we refer the reader to the third author's survey~\cite{vatter:permutation-cla:}, and review only what is necessary for our construction here.

Throughout this work, we view permutations in several slightly different ways, one of them being one-line notation. In this viewpoint, given permutations $\sigma=\sigma(1)\cdots\sigma(k)$ and $\pi=\pi(1)\cdots\pi(n)$, we say that $\sigma$ is \emph{contained} in $\pi$ if there are indices $1\le i_1<\cdots<i_k\le n$ such that the sequence $\pi(i_1)\cdots\pi(i_k)$ is in the same relative order as $\sigma$. If $\pi$ does not contain $\sigma$, then we say that it \emph{avoids} it. For us, a \emph{class} of permutations is a set of permutation closed downward under this containment order.

The \emph{permutation graph} of the permutation $\pi=\pi(1)\cdots\pi(n)$ is the graph $G_\pi$ on the vertices $\{1,\dots,n\}$ in which $i$ is adjacent to $j$ if and only if both $i<j$ and $\pi(i)>\pi(j)$. This mapping is many-to-one, as witnessed by the fact that $G_{2413}\cong G_{3142}\cong P_4$, for example. If $X$ is a set (or class) of permutations, then we denote by $G_X$ the set (or class) of permutation graphs corresponding to the members of $X$. A classic result of Dushnik and Miller~\cite{dushnik:partially-order:} states that a graph is (isomorphic to) a permutation graph if and only if it is both a comparability graph and a co-comparability graph. As shown by Gallai~\cite{gallai:transitiv-orien:}, the class of all permutation graphs is defined by infinitely many minimal forbidden induced subgraphs.

\begin{figure}
\begin{footnotesize}
\begin{center}
	\begin{tabular}{ccccccc}
	\begin{tikzpicture}[scale=0.2, baseline=(current bounding box.center)]
		\plotpermbox{0.5}{0.5}{8.5}{8.5};
		\plotpartialperm{1/3,4/8,5/5,7/1,8/4};
	\end{tikzpicture}
	&
	\begin{tikzpicture}[baseline=(current bounding box.center)]
		\node {$\le$};
	\end{tikzpicture}
	&
	\begin{tikzpicture}[scale=0.2, baseline=(current bounding box.center)]
		\plotpermbox{0.5}{0.5}{8.5}{8.5};
		\plotperm{3,6,2,8,5,7,1,4};
		\draw (1,3) circle (16pt);
		\draw (4,8) circle (16pt);
		\draw (5,5) circle (16pt);
		\draw (7,1) circle (16pt);
		\draw (8,4) circle (16pt);
	\end{tikzpicture}
	&
	\quad\quad\quad\quad
	&
	\begin{tikzpicture}[scale=0.2, baseline=(current bounding box.center)]
	    \draw (4,8) to [out=324.666666667, in=117] (7,1);
	    \draw (5,5) to [out=270, in=144] (7,1);
	    \draw (1,3) to [out=270, in=225] (7,1);
	    \draw (4,8) to [out=45, in=45] (8,4);
	    \draw (5,5) to [out=315, in=180] (8,4);
	    \draw (4,8) to [out=284.5, in=90] (5,5);
		\plotpermbox{0.5}{0.5}{8.5}{8.5};
		\plotpartialperm{1/3,4/8,5/5,7/1,8/4};
	\end{tikzpicture}
	&
	\begin{tikzpicture}[baseline=(current bounding box.center)]
		\node {$\le$};
	\end{tikzpicture}
	&
	\begin{tikzpicture}[scale=0.2, baseline=(current bounding box.center)]
	    \draw (6,7) to [out=300, in=90] (7,1);
	    \draw (4,8) to [out=324.666666667, in=117] (7,1);
	    \draw (5,5) to [out=270, in=144] (7,1);
	    \draw (2,6) to [out=315, in=171] (7,1);
	    \draw (3,2) to [out=315, in=198] (7,1);
	    \draw (1,3) to [out=270, in=225] (7,1);
	    \draw (4,8) to [out=45, in=45] (8,4);
	    \draw (6,7) to [out=330, in=90] (8,4);
	    \draw (2,6) to [out=45, in=135] (8,4);
	    \draw (5,5) to [out=315, in=180] (8,4);
	    \draw (1,3) to [out=0, in=135] (3,2);
	    \draw (2,6) to [out=270, in=90] (3,2);
	    \draw (2,6) to [out=0] (5,5);
	    \draw (4,8) to [out=284.5, in=90] (5,5);
	    \draw (4,8) to [out=4.833333333] (6,7);
		\plotpermbox{0.5}{0.5}{8.5}{8.5};
		\plotperm{3,6,2,8,5,7,1,4};
		\draw (1,3) circle (16pt);
		\draw (4,8) circle (16pt);
		\draw (5,5) circle (16pt);
		\draw (7,1) circle (16pt);
		\draw (8,4) circle (16pt);
	\end{tikzpicture}
	\\
	$25413$
	&
	$\le$
	&
	$36285714$
	&&
	$G_{25413}$
	&
	$\le$
	&
	$G_{36285714}$
	\end{tabular}
\end{center}
\end{footnotesize}
\caption{The containment order on permutations and their corresponding permutation graphs.}
\label{fig-perm-contain}
\end{figure}
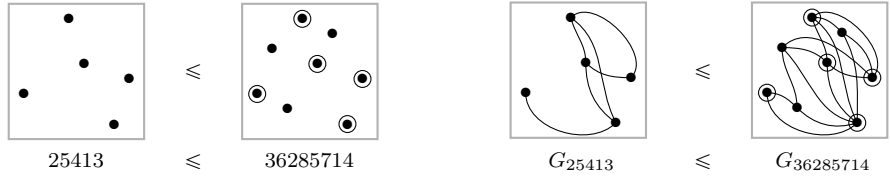

As done in Figure~\ref{fig-perm-contain}, we often identify a permutation $\pi$ with its \emph{plot}: the set of points $\{(i,\pi(i))\}$ in the plane. It is easy to check that if $\sigma$ is contained in $\pi$ then $G_\sigma$ is an induced subgraph of $G_\pi$, but the converse does not hold (returning to our example from above, $G_{2413}$ is an induced subgraph of $G_{3142}$ because the two graphs are isomorphic, but of course the permutation $2413$ is not contained in the permutation $3142$). For this reason, given a permutation class $\C$, it is stronger to show that $\C$ is wqo than to show that $G_{\C}$ is wqo, just as it is a stronger result that $G_\C$ is \emph{not} wqo than that $\C$ is not wqo.

When identifying permutations with their plots, it is clear that the permutation containment order respects all eight symmetries of the plane (which is isomorphic to the dihedral group on $8$ elements). Of these symmetries, three are particularly important to this work: the group-theoretic inverse, $\pi^{-1}$, which is obtained by reflecting the plot of $\pi$ about the line $y=x$; the reverse-complement, $\pi^{\text{rc}}$, obtained by reflecting the plot of $\pi$ about the line $y=-x$ (and then shifting); and the other symmetry obtained by composing these, $(\pi^{\text{rc}})^{-1}$. Note these symmetries do not affect the corresponding permutation graphs: for all permutations $\pi$, we have
\[
	G_\pi
	\cong
	G_{\pi^{-1}}
	\cong
	G_{\pi^{\text{rc}}}
	\cong
	G_{(\pi^{\text{rc}})^{-1}}.
\]
Thus for all permutation classes $\C$, we have
\[
	G_\C = G_{\C\cup\C^{-1}\cup\C^{\text{rc}}\cup (\C^{\text{rc}})^{-1}}.
\]

Sometimes we take a more liberal view and identify permutations with finite \emph{generic} sets of points in the plane, that is, with finite sets of points in the plane in which no two share the same $x$- or $y$-coordinate. If one labels a finite generic set of points in the plane from $1$ to $n$ by height (i.e., from bottom-to-top) and then records these labels reading left-to-right, a unique permutation is obtained.

\begin{figure}
\begin{center}
	$\pi\oplus\sigma=$
	\begin{tikzpicture}[scale=0.5, baseline=(current bounding box.center)]
		\plotpermbox{1}{1}{1}{1};
		\plotpermbox{2}{2}{2}{2};
		\node at (1,1) {$\pi$};
		\node at (2,2) {$\sigma$};
	\end{tikzpicture}
\quad\quad\quad\quad
	$\pi\ominus\sigma=$
	\begin{tikzpicture}[scale=0.5, baseline=(current bounding box.center)]
		\plotpermbox{1}{2}{1}{2};
		\plotpermbox{2}{1}{2}{1};
		\node at (1,2) {$\pi$};
		\node at (2,1) {$\sigma$};
	\end{tikzpicture}
\end{center}
\caption{The sum and skew sum operations on permutations.}
\label{fig-sums}
\end{figure}

We reference two specific permutation classes in this work. To define the first one we need the notions of the \emph{sum} and \emph{skew sum} of two permutations, which are pictorially defined in Figure~\ref{fig-sums}. More formally, if $\pi$ has length $k$ and $\sigma$ has length $\ell$, the sum of $\pi$ and $\sigma$ is the permutation defined as
\[
	(\pi\oplus\sigma)(i)
	=
	\left\{\begin{array}{ll}
	\pi(i)&\mbox{for $i\in[1,k]$},\\
	\sigma(i-k)+k&\mbox{for $i\in[k+1,k+\ell]$}.
	\end{array}\right.
\]
Similarly, the skew sum of these two permutations is defined is defined as
\[
	(\pi\ominus\sigma)(i)
	=
	\left\{\begin{array}{ll}
	\pi(i)+\ell&\mbox{for $i\in[1,k]$},\\
	\sigma(i-k)&\mbox{for $i\in[k+1,k+\ell]$}.
	\end{array}\right.
\]

A permutation that can be expressed as the sum (resp., skew sum) of two shorter permutations is said to be \emph{sum (resp., skew) decomposable}, and otherwise \emph{sum (resp., skew) indecomposable}. It is easy to see that every permutation $\pi$ can be expressed as the sum (resp., skew sum) of a sequence of sum (resp., skew) indecomposable permutations, that is, it can be written as $\pi=\alpha_1\oplus\cdots\oplus\alpha_m$ (resp., $\pi=\alpha_1\ominus\cdots\ominus\alpha_m$) where each $\alpha_i$ is sum (resp., skew) indecomposable (note that both sum and skew sum are associative operations, so there is no ambiguity in these expressions).

A permutation is \emph{separable} if it can be built from the permutation $1$ by repeated sums and skew sums. The set of separable permutations forms a permutation class. While this class first appeared in the work of Avis and Newborn~\cite{avis:on-pop-stacks-i:}, it was Bose, Buss, and Lubiw~\cite{bose:pattern-matchin:} who coined the term separable and proved that the minimal forbidden permutations for the class of separable permutations are $2413$ and $3142$. A graph theorist may recognize that this class of permutations is analogous to the class of \emph{complement-reducible graphs} (or \emph{co-graphs} for short), which are those graphs that can be built from $K_1$ via disjoint unions and stars. Indeed, a graph is a co-graph if and only if it is the permutation graph of a separable permutation.

Another permutation class we need is the class of \emph{skew-merged permutations}, which are those permutations whose entries can be partitioned into an increasing subsequence and a decreasing subsequence. This class was first introduced by Stankova~\cite{stankova:forbidden-subse:}, who proved that a permutation is skew-merged if and only if it avoids both $2143$ and $3412$. The analogous class of graphs is the class of \emph{split graphs}, whose vertices can be partitioned into an independent set and a clique. F\"oldes and Hammer~\cite{foldes:split-graphs:} proved that a graph is split if and only if it does not contain $2K_2$, $C_4$, or $C_5$ as induced subgraphs. Note that not all split graphs are permutation graphs of skew-merged permutations (for more details on this, we refer the reader to Corollary~\ref{cor-split-perm-graph-basis}).

Finally we need a term from the substitution decomposition of permutations (often called the modular decomposition in graph theory~\cite{gallai:transitiv-orien:}). An \emph{interval} in the permutation $\pi$ is a set of contiguous indices $I=[a,b]$ such that the set of values $\pi(I)=\{\pi(i) : i\in I\}$ is also contiguous. Every permutation of length $n$ has \emph{trivial} intervals of lengths $0$, $1$, and $n$; other intervals (if they exist) are called \emph{proper}.

\section{Widdershins Spirals}

Central to our counterexample to Conjecture~\ref{false-conj-lwqo} is a certain infinite family of permutations. In his thesis~\cite{murphy:restricted-perm:}, Murphy named these permutations widdershins spirals (widdershins being a Lower Scots word meaning ``to go anti-clockwise''). This name is explained by the drawing on the left of Figure~\ref{fig-widdershins}. Widdershins spirals are also a special type of pin sequence, as defined by Brignall, Huczynska, and Vatter~\cite{brignall:decomposing-sim:}, and it is this viewpoint we utilize to give a formal definition of them.

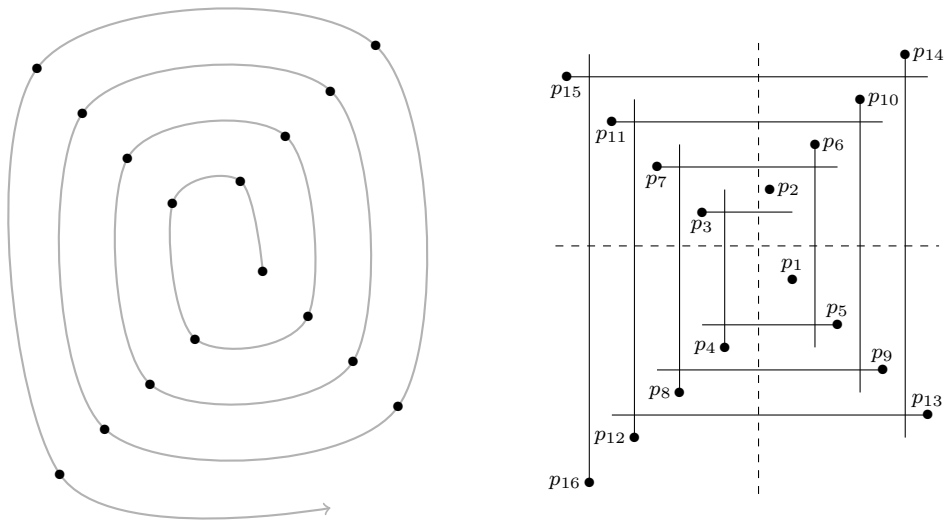
\begin{figure}
\begin{footnotesize}
\begin{center}
  \begin{tikzpicture}[scale=0.3, baseline=(current bounding box.center)]
    \draw [thick, darkgray, ->] plot [smooth, tension=0.6] coordinates {(2,-2) (1,2) (-2,1) (-1,-5) (4,-4) (3,4) (-4,3) (-3,-7) (6,-6) (5,6) (-6,5) (-5,-9) (8,-8) (7,8) (-8,7) (-7,-11) (5,-12.5)};
    \plotpartialperm{2/-2, 1/2, -2/1, -1/-5, 4/-4, 3/4, -4/3, -3/-7, 6/-6, 5/6, -6/5, -5/-9, 8/-8, 7/8, -8/7, -7/-11}
  \end{tikzpicture}
\quad\quad\quad\quad
  \begin{tikzpicture}[scale=0.3, baseline=(current bounding box.center)]
    \plotpartialperm{2/-2, 1/2, -2/1, -1/-5, 4/-4, 3/4, -4/3, -3/-7, 6/-6, 5/6, -6/5, -5/-9, 8/-8, 7/8, -8/7, -7/-11}
	\draw (-2,1)--(2,1);
	\draw (-1,-5)--(-1,2);
	\draw (4,-4)--(-2,-4);
	\draw (3,4)--(3,-5);
	\draw (-4,3)--(4,3);
	\draw (-3,-7)--(-3,4);
	\draw (6,-6)--(-4,-6);
	\draw (5,6)--(5,-7);
	\draw (-6,5)--(6,5);
	\draw (-5,-9)--(-5,6);
	\draw (8,-8)--(-6,-8);
	\draw (7,8)--(7,-9);
	\draw (-8,7)--(8,7);
	\draw (-7,-11)--(-7,8);
	\node at (2,-2) [above] {$p_1$};
	\node at (1,2) [right] {$p_2$};
	\node at (-2,1) [below] {$p_3$};
	\node at (-1,-5) [left] {$p_4$};
	\node at (4,-4) [above] {$p_5$};
	\node at (3,4) [right] {$p_6$};
	\node at (-4,3) [below] {$p_7$};
	\node at (-3,-7) [left] {$p_8$};
	\node at (6,-6) [above] {$p_9$};
	\node at (5,6) [right] {$p_{10}$};
	\node at (-6,5) [below] {$p_{11}$};
	\node at (-5,-9) [left] {$p_{12}$};
	\node at (8,-8) [above] {$p_{13}$};
	\node at (7,8) [right] {$p_{14}$};
	\node at (-8,7) [below] {$p_{15}$};
	\node at (-7,-11) [left] {$p_{16}$};
	\draw [dashed] (-8.5,-0.5)--(8.5,-0.5);
	\draw [dashed] (0.5,-11.5)--(0.5,8.5);
  \end{tikzpicture}
\end{center}
\end{footnotesize}
\caption{Two depictions of the same widdershins spiral of standard orientation.}
\label{fig-widdershins}
\end{figure}

An \emph{axis-parallel rectangle} is any rectangle in the plane in which all of its sides are parallel to the $x$- or $y$-axis. The \emph{rectangular hull} of a set of points in the plane is defined as the smallest axis-parallel rectangle containing them. Given a sequence of points $(p_1,\dots,p_i)$ in the plane, a \emph{proper pin} for this sequence is a point $p$ that lies outside their rectangular hull and \emph{separates} $p_i$ from $\{p_1,\dots,p_{i-1}\}$, meaning that $p$ lies either horizontally or vertically between $p_i$ and the rectangular hull of $\{p_1,\dots,p_{i-1}\}$. A \emph{proper pin sequence} is then constructed by starting with two points $p_1$ and $p_2$, choosing $p_3$ to be a proper pin for $(p_1,p_2)$, then choosing $p_4$ to be a proper pin for $(p_1,p_2,p_3)$, and so on. We describe pins as either \emph{left}, \emph{right}, \emph{up}, or \emph{down} based on their position relative to the rectangular hull of $\{p_1,\dots,p_{i-1}\}$. Note that the direction of a pin uniquely specifies its position relative to the previous points in a pin sequence.

Given this terminology, we may now define widdershins spirals. The \emph{widdershins spiral of standard orientation} is formed by starting with two points $p_1$ and $p_2$ such that $p_2$ lies to the northwest of $p_1$, and then taking any pin sequence of at least two more points where the positions of the pins $p_3$, $p_4$, $\dots$ constitute an initial segment of the repeating pattern left, down, right, up, left, down, right, up, $\dots$. The drawing on the right of Figure~\ref{fig-widdershins} shows a widdershins spiral from this perspective. We consider all rotations of a widdershins spiral of standard orientation by $90^\circ$, $180^\circ$, or $270^\circ$ to also be widdershins spirals. Note that all widdershins spirals are skew-merged permutations.

Denote by $\mathcal{W}$ the \emph{downward closure} of the set of widdershins spirals, i.e., $\mathcal{W}$ is the class of all permutations which are contained in some widdershins spiral. Because the widdershins spirals are skew-merged permutations, it follows that $\mathcal{W}$ is a subclass of the skew-merged permutations. Another notable property of $\mathcal{W}$ is that if $\pi\in\mathcal{W}$, then each of $1\oplus\pi$, $\pi\oplus 1$, $1\ominus\pi$, and $\pi\ominus 1$ are also members of $\mathcal{W}$. Finally, note that $\mathcal{W}$ is closed under rotations by $90^\circ$, $180^\circ$, and $270^\circ$.

We now introduce a decomposition for members of $\mathcal{W}$. First suppose that the permutation $\pi\in\mathcal{W}$ is sum decomposable, and thus can be written as $\pi=\sigma\oplus\tau$ for nonempty permutations $\sigma$ and $\tau$. Because $\pi\in\mathcal{W}$, it is skew-merged, and thus avoids $2143$. Therefore at most one of $\sigma$ and $\tau$ may contain the permutation $21$, so at least one of the $\sigma$ or $\tau$ must be increasing. It follows that every sum decomposable permutation in $\mathcal{W}$ can be expressed as either $1\oplus\alpha$ or $\alpha\oplus 1$. Since all permutations in $\mathcal{W}$ avoid $3412$, a similar argument shows that all skew decomposable permutations in this class can be expressed as either $1\ominus\alpha$ or $\alpha\ominus 1$.

\begin{figure}
\begin{center}
  \begin{tikzpicture}[scale=0.15, baseline=(current bounding box.center)]
    \draw [white] (0,-4)--(0,8); 
    \draw [darkgray, thick, line cap=round] (0,0) rectangle (5,5);
	\node at (2.5, 2.5) {$\alpha$};
    \plotpartialperm{-1/-1};
  \end{tikzpicture}
\quad\quad
  \begin{tikzpicture}[scale=0.15, baseline=(current bounding box.center)]
    \draw [white] (0,-4)--(0,8); 
    \draw [darkgray, thick, line cap=round] (0,0) rectangle (5,5);
	\node at (2.5, 2.5) {$\alpha$};
    \plotpartialperm{6/6};
  \end{tikzpicture}
\quad\quad
  \begin{tikzpicture}[scale=0.15, baseline=(current bounding box.center)]
    \draw [white] (0,-4)--(0,8); 
    \draw [darkgray, thick, line cap=round] (0,0) rectangle (5,5);
	\node at (2.5, 2.5) {$\alpha$};
    \plotpartialperm{-1/6};
  \end{tikzpicture}
\quad\quad
  \begin{tikzpicture}[scale=0.15, baseline=(current bounding box.center)]
    \draw [white] (0,-4)--(0,8); 
    \draw [darkgray, thick, line cap=round] (0,0) rectangle (5,5);
	\node at (2.5, 2.5) {$\alpha$};
    \plotpartialperm{6/-1};
  \end{tikzpicture}
\quad\quad
  \begin{tikzpicture}[scale=0.15, baseline=(current bounding box.center)]
    \draw [white] (0,-3)--(0,8); 
    \draw [thick, darkgray, ->] plot [smooth, tension=0.6] coordinates {(-1,-2) (7,-1) (6,7) (-3,6) (-2,-4)};
    \plotpartialperm{-1/-2, 7/-1, 6/7, -3/6}
    \draw [darkgray, thick, line cap=round] (0,0) rectangle (5,5);
	\node at (2.5, 2.5) {$\alpha$};
  \end{tikzpicture}
\end{center}
\caption{The five possibilities in the ring decomposition of members of $\mathcal{W}$.}
\label{fig-ring-decomp}
\end{figure}
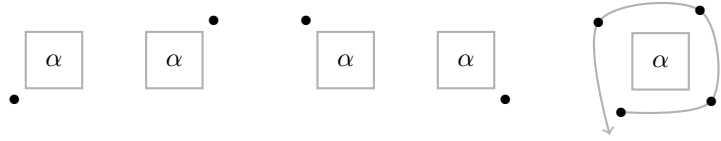

To handle members of $\mathcal{W}$ which are neither sum nor skew decomposable we introduce an operation we refer to as \emph{central insertion}. If one thinks of a widdershins spiral of standard orientation as being centered at the origin of the plane, with the $x$-axis lying vertically between $p_1$ and $p_3$ and the $y$-axis lying horizontally between $p_4$ and $p_2$ (these are drawn as dashed lines on the right of Figure~\ref{fig-widdershins}), then this operation consists of placing another permutation $\alpha$ (also from $\mathcal{W}$) at the origin and shifting the other points of the widdershin spiral appropriately (as shown on the right of Figure~\ref{fig-ring-decomp}). In doing so, the entries of $\alpha$ become an interval in the resulting permutation. We define central insertion into widdershins spirals of other orientations symmetrically.

Suppose that the permutation $\pi\in\mathcal{W}$ is neither sum nor skew decomposable and consider a shortest widdershins spiral $\Psi$ containing $\pi$. Label the points of $\Psi$ as $p_1$, $p_2$, $\dots$, $p_k$ as above, and fix an embedding of $\pi$ into $\Psi$. If $\pi$ is not equal to $\Psi$, then consider the outer-most point deleted from $\Psi$ to form $\pi$. This outer-most point cannot be $p_k$, because that would contradict the minimality of $\Psi$. If this outer-most point is $p_{k-1}$, $p_{k-2}$, or $p_{k-3}$, then $\pi$ is either sum or skew decomposable, another contradiction. If this outer-most point is $p_1$, then $\pi$ is still a widdershins spiral, again contradicting the minimality of $\Psi$. In all other cases, it follows from the construction of the widdershins spiral that $\pi$ contains a unique maximal proper interval, which lies at the center of the spiral, as shown on the right of Figure~\ref{fig-ring-decomp}.

These observations establish the following result, which we refer to as the \emph{ring decomposition} of the class $\mathcal{W}$.

\begin{proposition}
\label{prop-ring-decomp}
Every permutation in $\mathcal{W}$ can be expressed either as $1\oplus\alpha$, $\alpha\oplus 1$, $1\ominus\alpha$, and $\alpha\ominus 1$ for $\alpha\in\mathcal{W}$ or as the central insertion into a widdershins spiral of length at least $4$ by a (possibly empty) member of $\mathcal{W}$.
\end{proposition}

Note that there may be several different ring decompositions of a member of $\mathcal{W}$; for example, $(1\oplus\alpha)\oplus 1=1\oplus(\alpha\oplus 1)$. With a bit more effort, however, one can obtain a unique ring decomposition of each member of $\mathcal{W}$, which can produce the generating function for this class. We state this generating function below, omitting its proof because the result is not relevant to the goals of this paper.

\begin{corollary}
The generating function of the downward closure of the widdershins spirals is
\[
    \frac{1-4x+3x^2}{1-5x+6x^2-2x^3-x^4-3x^5}.
\]
%
%
\end{corollary}

The ring decomposition of $\mathcal{W}$ also allows us to show that the class is wqo, but first we must recall Higman's Lemma. Suppose that $(\Sigma,\le)$ is any poset. Then the \emph{generalized subword order} on $\Sigma^\ast$ is defined by $v\le w$ if there are indices $1\le i_1<i_2<\cdots<i_{|v|}\le|w|$ such that $v(j)\le w(i_j)$ for all $j$.

\begin{higmans-lemma*}
If $(\Sigma,\le)$ is well quasi-ordered then $\Sigma^*$, ordered by the generalized subword order, is also well quasi-ordered.
\end{higmans-lemma*}

Given any ring decomposition of a permutation in $\mathcal{W}$, we now describe how to encode the decomposition as a word over the alphabet
\[
	\Sigma
	=
	\{ \mathsf{NE}_k,\ \mathsf{NW}_k,\ \mathsf{SW}_k,\ \mathsf{SE}_k
		\st
		\text{$k=1$ or $k\ge 4$}
	\}.
\]
Our goal is to give a partial order for $\Sigma$ that is a wqo, and construct an encoding $w_\pi\in\Sigma^\ast$ of every permutation $\pi\in\mathcal{W}$ so that if $w_\sigma$ is a generalized subword of $w_\pi$ then $\sigma$ is a subpermutation of $\pi$. Higman's Lemma will then allow us to conclude that $\mathcal{W}$ is itself wqo.

To form this encoding, first fix one ring decomposition of every member of $\mathcal{W}$. (Each member has at least one such decomposition by Proposition~\ref{prop-ring-decomp}, and for the purposes of this argument we do not need to specify which is chosen.) The encoding of the chosen ring decomposition of the permutation $\pi\in\mathcal{W}$ is denoted $w_\pi$. First, if $\pi$ is of the form $1\oplus\alpha$, then $w_\pi=w_\alpha\,\mathsf{SW}_1$, which denoted that $\pi$ is obtained by adding a single point in the southwest quadrant to $\alpha$. Similarly, if $\pi$ is of the form $\alpha\oplus 1$, $1\ominus\alpha$, or $\alpha\ominus 1$, respectively, then $\pi$ is encoded as $w_\alpha\,\mathsf{NE}_1$, $w_\alpha\,\mathsf{NW}_1$, or $w_\alpha\,\mathsf{SE}_1$. Finally, if $\pi$ is formed by inserting a permutation $\alpha\in\mathcal{W}$ into the center of a widdershins spiral of length at least $4$, then we encode $\pi$ as either $w_\alpha\,\mathsf{NE}_k$, $w_\alpha\,\mathsf{NW}_k$, $w_\alpha\,\mathsf{SW}_k$, or $w_\alpha\,\mathsf{SE}_k$, where $k$ denotes the length of the widdershins spiral and the direction ($\mathsf{NE}$, $\mathsf{NW}$, $\mathsf{SW}$, or $\mathsf{SE}$) denotes the position of the first point of the widdershins spiral relative to $\alpha$.

We now place a partial order on our alphabet $\Sigma$, which is based on the permutation containment order. While one might note that no matter what quadrant the first point of a widdershins spiral lies in, it is contained in all widdershins spirals of length at least $3$ greater, this observation is not necessary for our proof. Instead, we simply need to note that all widdershins spirals with their first point in a particular quadrant form a chain, and thus we can take our partial order on $\Sigma$ to be simply a union of four chains: $\mathsf{NE}_k\le\mathsf{NE}_\ell$ if $k\le\ell$, and similarly for the other three quadrants.

Clearly $\Sigma$ is wqo under this partial order, and moreover, it is clear that if $w_{\sigma}$ is a generalized subword of $w_{\pi}$ then $\sigma$ is a subpermutation of $\pi$. Thus Higman's Lemma implies that $\mathcal{W}$ is itself wqo.

\begin{proposition}
\label{prop-W-wqo}
The downward closure of the widdershins spirals is wqo.
\end{proposition}
%

Our next goal is to determine the minimal forbidden permutations for $\mathcal{W}$, which will later be essential for determining the minimal forbidden induced subgraphs of $G_{\mathcal{W}}$.

\begin{proposition}
\label{prop-W-basis}
The downward closure of the widdershins spirals is defined by the minimal forbidden permutations $2143$, $2413$, $3412$, $314562$, $412563$, $415632$, $431562$, $512364$, $512643$, $516432$, $541263$, $541632$, and $543162$.
%
%
\end{proposition}
\begin{proof}
Let $B$ denote the list of forbidden permutations stated in the proposition, and let $\C$ denote the permutation class that has $B$ as its set of minimal forbidden permutations. We aim to show that $\mathcal{W}=\C$. It can be seen by inspection that no element of $B$ lies in $\mathcal{W}$, so we may conclude that $\mathcal{W}$ is contained in $\C$. We prove by induction on the length of $\pi$ that every $\pi\in\C$ lies in $\mathcal{W}$. The claim is clearly true for $|\pi|\le 2$, so now consider a longer permutation $\pi\in\C$. Note that $B$, and hence also $\C$ (whose minimal forbidden permutations are precisely $B$), are closed under rotations by $90^\circ$.

First suppose that $\pi$ is sum or skew decomposable. As both cases are equivalent up to rotational symmetry, we may assume that $\pi=\sigma\oplus\tau$ for nonempty permutations $\sigma$ and $\tau$. Because $\pi$ avoids $21\oplus 21=2143$, at most one of these components may contain $21$, and thus the other must be increasing. By induction, we may conclude that $\pi$ lies in the downward closure of the widdershins spirals, as desired.

\begin{figure}  
\begin{center}
	\begin{tikzpicture}[scale=0.5, baseline=(current bounding box.center)]
		\draw [darkgray, fill=lightgray] (0,2) rectangle (1,4);
		\draw [darkgray, fill=lightgray] (1,0) rectangle (3,1);
		\draw [darkgray, fill=lightgray] (2,4) rectangle (4,5);
		\draw [darkgray, fill=lightgray] (4,1) rectangle (5,3);
		\draw [darkgray, pattern=north west lines, pattern color=darkgray] (0,1) rectangle (2,2);
		\draw [darkgray, pattern=north west lines, pattern color=darkgray] (1,3) rectangle (2,5);
		\draw [darkgray, pattern=north west lines, pattern color=darkgray] (3,0) rectangle (4,2);
		\draw [darkgray, pattern=north west lines, pattern color=darkgray] (3,3) rectangle (5,4);
		\foreach \i in {0,...,5} {
			\draw [darkgray, thick, line cap=round] (0,\i)--(5,\i);
			\draw [darkgray, thick, line cap=round] (\i,0)--(\i,5);
		}
		\plotperm{3,1,4,2};
		\node[rotate=45] at (0.55,0.55) {$\dots$};
		\node[rotate=135] at (4.45,0.55) {$\dots$};
		\node[rotate=45] at (4.5,4.55) {$\dots$};
		\node[rotate=135] at (0.5,4.55) {$\dots$};
	\end{tikzpicture}
\quad\quad\quad
	\begin{tikzpicture}[scale=0.5, baseline=(current bounding box.center)]
		\draw [darkgray, fill=lightgray] (0,2) rectangle (1,4);
		\draw [darkgray, fill=lightgray] (1,0) rectangle (3,1);
		\draw [darkgray, fill=lightgray] (2,4) rectangle (4,5);
		\draw [darkgray, fill=lightgray] (4,1) rectangle (5,3);
		\draw [darkgray, fill=lightgray] (0,1) rectangle (2,2);
		\draw [darkgray, fill=lightgray] (1,3) rectangle (2,5);
		\draw [darkgray, fill=lightgray] (3,0) rectangle (4,2);
		\draw [darkgray, fill=lightgray] (3,3) rectangle (5,4);
		\draw [darkgray, pattern=north west lines, pattern color=darkgray] (1,2) rectangle (2,3);
		\draw [darkgray, pattern=north west lines, pattern color=darkgray] (2,1) rectangle (3,2);
		\draw [darkgray, pattern=north west lines, pattern color=darkgray] (2,3) rectangle (3,4);
		\draw [darkgray, pattern=north west lines, pattern color=darkgray] (3,2) rectangle (4,3);
		\foreach \i in {0,...,5} {
			\draw [darkgray, thick, line cap=round] (0,\i)--(5,\i);
			\draw [darkgray, thick, line cap=round] (\i,0)--(\i,5);
		}
		\plotperm{3,1,4,2};
		\draw [darkgray, fill=lightgray] (0,0) rectangle (1,1);
		\draw [darkgray, fill=lightgray] (0,4) rectangle (1,5);
		\draw [darkgray, fill=lightgray] (4,0) rectangle (5,1);
		\draw [darkgray, fill=lightgray] (4,4) rectangle (5,5);
	\end{tikzpicture}
\quad\quad\quad
	\begin{tikzpicture}[scale=0.5, baseline=(current bounding box.center)]
		\draw [darkgray, fill=lightgray] (0,2) rectangle (1,4);
		\draw [darkgray, fill=lightgray] (1,0) rectangle (3,1);
		\draw [darkgray, fill=lightgray] (2,4) rectangle (4,5);
		\draw [darkgray, fill=lightgray] (4,1) rectangle (5,3);
		\draw [darkgray, fill=lightgray] (0,1) rectangle (2,2);
		\draw [darkgray, fill=lightgray] (1,3) rectangle (2,5);
		\draw [darkgray, fill=lightgray] (3,0) rectangle (4,2);
		\draw [darkgray, fill=lightgray] (3,3) rectangle (5,4);
		\draw [darkgray, fill=lightgray] (1,2) rectangle (2,3);
		\draw [darkgray, fill=lightgray] (2,1) rectangle (3,2);
		\draw [darkgray, fill=lightgray] (3,2) rectangle (4,3);
		\draw [darkgray, fill=lightgray] (2,3) rectangle (3,4);
		\draw [darkgray, pattern=north west lines, pattern color=darkgray] (2.5,2) rectangle (3,3);
		\foreach \i in {0,...,5} {
			\draw [darkgray, thick, line cap=round] (0,\i)--(5,\i);
			\draw [darkgray, thick, line cap=round] (\i,0)--(\i,5);
		}
		\plotperm{3,1,4,2};
		\draw (2.5,2)--(2.5,3.5);
		\absdot{(2.5,3.5)}{};
		\draw [darkgray, fill=lightgray] (0,0) rectangle (1,1);
		\draw [darkgray, fill=lightgray] (0,4) rectangle (1,5);
		\draw [darkgray, fill=lightgray] (4,0) rectangle (5,1);
		\draw [darkgray, fill=lightgray] (4,4) rectangle (5,5);
	\end{tikzpicture}
\quad\quad\quad
	\begin{tikzpicture}[scale=0.5, baseline=(current bounding box.center)]
		\draw [darkgray, fill=lightgray] (0,2) rectangle (1,4);
		\draw [darkgray, fill=lightgray] (1,0) rectangle (3,1);
		\draw [darkgray, fill=lightgray] (2,4) rectangle (4,5);
		\draw [darkgray, fill=lightgray] (4,1) rectangle (5,3);
		\draw [darkgray, fill=lightgray] (0,1) rectangle (2,2);
		\draw [darkgray, fill=lightgray] (1,3) rectangle (2,5);
		\draw [darkgray, fill=lightgray] (3,0) rectangle (4,2);
		\draw [darkgray, fill=lightgray] (3,3) rectangle (5,4);
		\draw [darkgray, fill=lightgray] (1,2) rectangle (2,3);
		\draw [darkgray, fill=lightgray] (2,1) rectangle (3,2);
		\draw [darkgray, fill=lightgray] (3,2) rectangle (4,3);
		\draw [darkgray, fill=lightgray] (2,3) rectangle (3,4);
		\draw [darkgray, fill=lightgray] (2.5,2) rectangle (3,3);
		\draw [darkgray, pattern=north west lines, pattern color=darkgray] (2,2) rectangle (2.5,2.5);
		\foreach \i in {0,...,5} {
			\draw [darkgray, thick, line cap=round] (0,\i)--(5,\i);
			\draw [darkgray, thick, line cap=round] (\i,0)--(\i,5);
		}
		\plotperm{3,1,4,2};
		\draw (2.5,2)--(2.5,3.5);
		\absdot{(2.5,3.5)}{};
		\draw (2.75,2.5)--(2,2.5);
		\absdot{(2.75,2.5)}{};
		\draw [darkgray, fill=lightgray] (0,0) rectangle (1,1);
		\draw [darkgray, fill=lightgray] (0,4) rectangle (1,5);
		\draw [darkgray, fill=lightgray] (4,0) rectangle (5,1);
		\draw [darkgray, fill=lightgray] (4,4) rectangle (5,5);
	\end{tikzpicture}
\end{center}
\caption{The analysis arising in the proof of Proposition~\ref{prop-W-basis}.}
\label{fig-cell-diagram}
\end{figure}
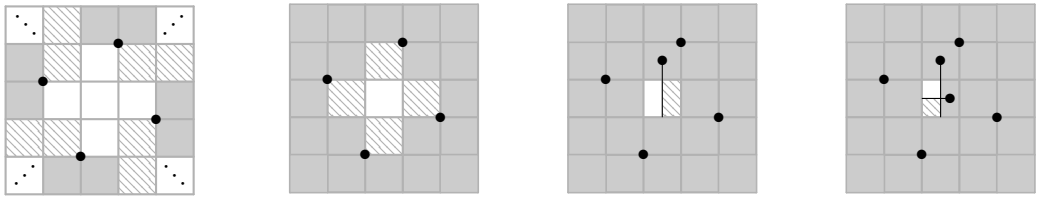

If $\pi$ avoids $3142$, then it is separable (because it also avoids $2413\in B$), and thus $\pi$ must be sum or skew decomposable, and both cases are handled by the argument above. It remains to consider the case where $\pi$ contains $3142$. Among all copies of $3142$ in $\pi$, choose one with the $3$ as far west as possible, the $1$ as far south as possible, the $4$ as far north as possible, and the $2$ as far east as possible. These choices ensure that no points lie in the shaded cells of the first panel of Figure~\ref{fig-cell-diagram}. Moreover, as $\pi$ avoids $2143$ and $3412$, no points may lie in the hatched cells of the same panel of this figure. Finally, note that by our choice of the copy of $3142$ to examine, if $\pi$ has any entries in the four corners labelled by ellipses in the first panel of Figure~\ref{fig-cell-diagram} then these entries are themselves order isomorphic a separable permutation. Thus if these cells are nonempty, we are done by induction and our previous argument regarding sum or skew decomposable permutations.

We may now move to the second panel of Figure~\ref{fig-cell-diagram}. The hatched cells of this panel may contain at most one entry in total, as otherwise, if $\pi$ contained two points in the same hatched cell then $\pi$ would contain $2143$, $3412$, $543162$, $512364$, $314562$, or $516432$ (all members of $B$), while if $\pi$ contained two points in different hatched cells then $\pi$ would contain $412563$, $415632$, $431562$, $512643$, $541263$, or $541632$ (also members of $B$).

If $\pi$ does not contain any points in the hatched cells of the second panel of Figure~\ref{fig-cell-diagram}, then $\pi$ is obtained by central inflation of a widdershins spiral, and we are done by induction. Thus we may assume that $\pi$ contains precisely one point in these hatched cells, and by rotational symmetry, we may suppose that this point lies in the northern-most hatched cell. This leads us to the situation displayed in the third panel of Figure~\ref{fig-cell-diagram}; note that $\pi$ can contain at most one point in the hatched cell of this diagram because otherwise it would contain $3412$ or $516432$. If $\pi$ contains no points in this hatched cell, then we are done by induction. If $\pi$ does contain a point in this hatched cell, we are left in the situation shown in the fourth and final panel of Figure~\ref{fig-cell-diagram}. This process cannot continue indefinitely, and when it terminates, we may conclude by induction that $\pi\in\mathcal{W}$.
\end{proof}

%
%
%

Recall from our discussion in Section~\ref{sec-perm-graphs} that the class of permutation graphs corresponding to $\mathcal{W}$ also contains the permutation graphs of members of the classes $\mathcal{W}^{-1}$, $\mathcal{W}^{\text{rc}}$, and $(\mathcal{W}^{\text{rc}})^{-1}$. Since $\mathcal{W}$ (and thus also $\mathcal{W}^{-1}$) is closed under rotation by $90^{\circ}$, and the reverse-complement operation is the same as a rotation by $180^\circ$, we see that $\mathcal{W}^{\text{rc}}=\mathcal{W}$ and $(\mathcal{W}^{\text{rc}})^{-1}=\mathcal{W}^{-1}$. Therefore we have
\[
	G_{\mathcal{W}} = G_{\mathcal{W}\cup\mathcal{W}^{-1}}.
\]
Thus instead of the minimal forbidden permutations of $\mathcal{W}$, we are more interested in the minimal forbidden permutations of $\mathcal{W}\cup\mathcal{W}^{-1}$. It is straight-forward to compute the minimal forbidden permutations of $\mathcal{W}^{-1}$ using Proposition~\ref{prop-W-basis}. Moreover, one of the results of Atkinson's seminal work on permutation patterns~\cite{atkinson:restricted-perm:} describes how to compute the minimal forbidden permutations of a union of two permutation classes. In our context, because the minimal forbidden permutations for both $\mathcal{W}$ and $\mathcal{W}^{-1}$ have length at most $6$, Atkinson's result shows that to compute the minimal forbidden permutations of $\mathcal{W}\cup\mathcal{W}^{-1}$ we may restrict our attention to permutations of length at most $12$. A computer search then yields the following result. 

\begin{corollary}
\label{cor-WcupWinv-basis}
The class $\mathcal{W}\cup\mathcal{W}^{-1}$ is defined by the minimal forbidden permutations
$2143$, 
$3412$, 
$234615$, 
$236145$, 
$236514$, 
$261345$, 
$265134$, 
$265413$, 
$314562$, 
$346215$, 
$362145$, 
$365214$, 
$412563$, 
$415632$, 
$431562$, 
$463215$, 
$512364$, 
$512643$, 
$516432$, 
$541263$, 
$541632$, 
$543162$, 
$28536417$, and 
$71463582$.
\end{corollary}

%

%
%
%
%
%

\section{The Counterexample to Conjecture~\ref{false-conj-lwqo}}

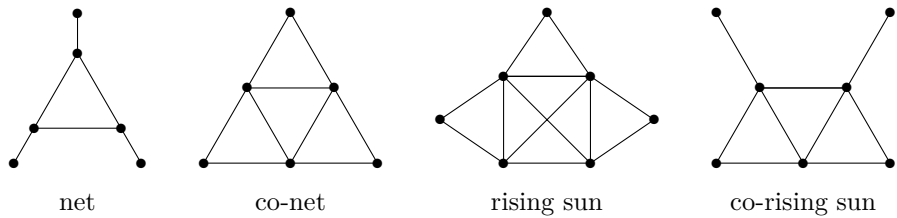
\begin{figure}
\begin{center}
\begin{tabular}{cccc}
	\begin{tikzpicture}[scale=1.15485, yscale=-1]
		\plotpartialperm{0/0, -0.5/0.866, 0.5/0.866, 0/-0.464, -0.732/1.267824, 0.732/1.267824};
		\draw (0,0)--(-0.5,0.866)--(0.5,0.866)--cycle;
		\draw (-0.5,0.866)--(-0.732,1.267824);
		\draw (0.5,0.866)--(0.732,1.267824);
		\draw (0,-0.464)--(0,0);
	\end{tikzpicture}
	&
	\begin{tikzpicture}[scale=1.15485]
		\plotpartialperm{-1/0, 0/0, 1/0, -0.5/0.866, 0.5/0.866, 0/1.732};
		\draw (-1,0)--(0,1.732)--(1,0)--cycle;
		\draw (0,0)--(-0.5,0.866)--(0.5,0.866)--cycle;
	\end{tikzpicture}
	&
	\begin{tikzpicture}[scale=1.15485]
		\plotpartialperm{0/0, 1/0, 1/1, 0/1, 0.5/1.732, -0.732/0.5, 1.732/0.5};
		\draw (0,0)--(1,0)--(1,1)--(0,1)--(0,0)--(1,1)--(0,1)--(1,0);
		\draw (0,0)--(-0.732,0.5)--(0,1)--(0.5,1.732)--(1,1)--(1.732,0.5)--(1,0);
	\end{tikzpicture}
	&
	\begin{tikzpicture}[scale=1.15485]
		\plotpartialperm{-1/0, 0/0, 1/0, -0.5/0.866, 0.5/0.866, 1/1.732, -1/1.732};
		\draw (0,0)--(-0.5,0.866)--(0.5,0.866)--cycle;
		\draw (-1,0)--(1,0)--(0.5,0.866)--(-0.5,0.866)--(-1,0);
		\draw (0.5,0.866)--(1,1.732);
		\draw (-0.5,0.866)--(-1,1.732);
	\end{tikzpicture}
\\
	net
	&
	co-net
	&
	rising sun
	&
	co-rising sun
\end{tabular}
\end{center}
\caption{In addition to $2K_2$, $C_4$, $C_5$, the forbidden induced subgraphs for the split permutation graphs.}
\label{fig-split-perm-basis}
\end{figure}

We now turn our attention to the permutation graphs in $G_{\mathcal{W}}$, seeking first to determine the (finite list of) minimal forbidden induced subgraphs for this class, and then to show that it is not lwqo. Throughout this section we use the graph naming conventions from the \emph{Information System on Graph Classes and their Inclusions}~\cite{information-system-on-graph-classes-and-their-inclusions:15:}.

F\"oldes and Hammer~\cite{foldes:split-graphs-ha:} established the following result in 1977. The more exotic graphs involved in the statements of both Theorem~\ref{thm-split-compare-basis} and Corollary~\ref{cor-split-perm-graph-basis} are depicted in Figure~\ref{fig-split-perm-basis}.

\begin{theorem}[F\"oldes and Hammer~\cite{foldes:split-graphs-ha:}]
\label{thm-split-compare-basis}
The class of split comparability graphs is defined by the minimal forbidden induced subgraphs $2K_2$, $C_4$, $C_5$, net, co-net, and co-rising sun.
\end{theorem}

It follows that the class of split co-comparability graphs is thus defined by the minimal forbidden subgraphs $2K_2$, $C_4$, $C_5$, net, co-net, and rising sun. As the class of permutation graphs is the intersection of the classes of comparability and co-comparability graphs (Dushnik and Miller~\cite{dushnik:partially-order:}), we obtain the minimal forbidden induced subgraph characterization of the split permutation graphs.

\begin{corollary}
\label{cor-split-perm-graph-basis}
The class of split permutation graphs is defined by the minimal forbidden induced subgraphs $2K_2$, $C_4$, $C_5$, net, co-net, rising sun and co-rising sun.
\end{corollary}

Corollary~\ref{cor-split-perm-graph-basis} also follows from the 1985 work of Benzaken, Hammer, and de Werra~\cite{benzaken:split-graphs-of:}, as described in Brandst\"adt, Le, and Spinrad~\cite[Theorem 7.1.2]{brandstadt:graph-classes:-:}. 

Split permutation graphs have themselves received some attention relating to wqo. Korpelainen, Lozin, and Mayhill~\cite{korpelainen:split-permutati:} established that this class contains an (unlabelled) infinite antichain, and Atminas, Brignall, Lozin, and Stacho~\cite{atminas:minimal-classes:} showed that it contains a \emph{canonical} labelled infinite antichain. This means that a subclass of the split permutation graphs is lwqo if and only it has finite intersection with the (unlabelled) graphs contained in this antichain. 

We can now characterize the minimal forbidden subgraphs of $G_\mathcal{W}$.

\begin{figure}
\begin{center}
\begin{tabular}{ccccccc}
	\begin{tikzpicture}[xscale=0.5, yscale=1]
		\plotpartialperm{-1/-1, -1/0, -1/1, 1/-1, 1/0, 1/1};
		\draw (-1,-1)--(-1,1);
		\draw (1,-1)--(1,1);
		\draw (-1,0)--(1,0);
	\end{tikzpicture}
	&
	\begin{tikzpicture}[xscale=0.5, yscale=1]
		\plotpartialperm{0/-1, -1/-0.5, -1/0.5, 1/-0.5, 1/0.5, 0/1};
		\draw (-1,-0.5)--(0,-1)--(1,-0.5)--cycle;
		\draw (-1,0.5)--(0,1)--(1,0.5)--cycle;
		\draw (-1,-0.5)--(1,-0.5)--(1,0.5)--(-1,0.5)--cycle;
		\draw (-1,-0.5)--(1,0.5);
		\draw (1,-0.5)--(-1,0.5);
	\end{tikzpicture}
	&
	\begin{tikzpicture}[xscale=0.5, yscale=1]
		\plotpartialperm{-1/0.33333, 0/0.33333, 1/0.33333, 0/-1, 0/-0.33333, 0/0.33333, 0/1};
		\draw (-1,0.33333)--(1,0.33333);
		\draw (0,-1)--(0,1);
	\end{tikzpicture}
	&
	\begin{tikzpicture}[xscale=0.5, yscale=1]
		\plotpartialperm{-1/1, 1/1, 0/-1, 0/-0.5, 0/0, 0/0.5};
		\draw (-1,1)--(1,1);
		\draw (-1,1)--(0,0.5)--(1,1);
		\draw (-1,1)--(0,0)--(1,1);
		\draw (-1,1)--(0,-0.5)--(1,1);
		\draw (0,-1)--(0,0.5);
	\end{tikzpicture}
	&
	\begin{tikzpicture}[xscale=0.5, yscale=1]
		\plotpartialperm{-1/-1, 1/-1, 0/-0.33333, -1/0.33333, 1/0.33333, 0/1};
		\draw (-1,-1)--(0,-0.33333)--(-1,0.33333);
		\draw (1,-1)--(0,-0.33333)--(1,0.33333);
		\draw (-1,0.33333)--(1,0.33333)--(0,1)--cycle;
	\end{tikzpicture}
	&
	\begin{tikzpicture}[xscale=0.5, yscale=1]
		\plotpartialperm{-1/1, 1/1, 0/-1, 0/-0.5, 0/0, 0/0.5};
		\draw (-1,1)--(1,1);
		\draw (-1,1)--(0,0.5)--(1,1);
		\draw (-1,1)--(0,0)--(1,1);
		\draw (-1,1)--(0,-0.5)--(1,1);
		\draw (0,-1)--(0,-0.5);
	\end{tikzpicture}
	&
	\begin{tikzpicture}[xscale=0.5, yscale=1]
		\plotpartialperm{0/-1, 0/-0.5, 0/0.5, 0/1, -1/0, -0.5/0, 0.5/0, 1/0};
		\draw (0,-1)--(0,1);
		\draw (-1,0)--(1,0);
		\draw (0,-0.5)--(-1,0);
		\draw (0,-0.5)--(-0.5,0);
		\draw (0,-0.5)--(0.5,0);
		\draw (0,-0.5)--(1,0);
		\draw (0,0.5)--(-1,0);
		\draw (0,0.5)--(-0.5,0);
		\draw (0,0.5)--(0.5,0);
		\draw (0,0.5)--(1,0);
	\end{tikzpicture}
\\
	$H$
	&
	$\overline{H}$
	&
	cross
	&
	co-cross
	&
	$X_{168}$
	&
	$\overline{X_{168}}$
	&
	$X_{160}\cong \overline{X_{160}}$
\end{tabular}
\end{center}
\caption{In addition to $2K_2$, $C_4$, $C_5$, net, co-net, rising sun, and co-rising sun, the forbidden induced subgraphs for $G_{\mathcal{W}}$.}
\label{fig-W-graphs-basis}
\end{figure}
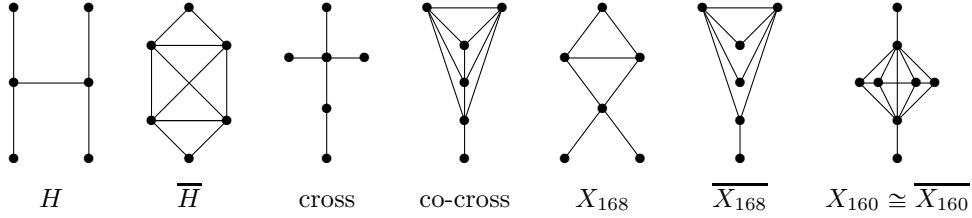

\begin{proposition}
\label{prop-basis-GW}
The class $G_{\mathcal{W}}$ is defined by the minimal forbidden induced subgraphs $2K_2$, $C_4$, $C_5$, net, co-net, rising sun, co-rising sun, $H$, $\overline{H}$, cross, co-cross, $X_{168}$, $\overline{X_{168}}$, and $X_{160}$.
\end{proposition}
\begin{proof}
Corollary~\ref{cor-WcupWinv-basis} gives the minimal forbidden permutations of $\mathcal{W}\cup\mathcal{W}^{-1}$. Our first step is to compute the permutation graphs corresponding to these permutations, and then to compute all permutations corresponding to those permutation graphs. A brute-force computation yields the following chart. Observe that the permutations on the left-hand column of this chart are precisely the minimal forbidden permutations for $\mathcal{W}\cup\mathcal{W}^{-1}$. 
\[
	\begin{array}{ll}
	\beta&G_\beta\\[2pt]\hline\\[-8pt]
	2143		&	2K_2\\[2pt]
	3412		&	C_4\\[2pt]
	236145, 412563	&	H\\[2pt]
	365214, 541632	&	\overline{H}\\[2pt]
	234615, 261345, 314562, 512364	&	\text{cross}\\[2pt]
	265413, 463215, 516432, 543162	&	\text{co-cross}\\[2pt]
	236514, 362145, 431562, 512643	&	X_{168}\\[2pt]
	265134, 346215, 415632, 541263	&	\overline{X_{168}}\\[2pt]
	28536417, 71463582	&	X_{160}
	\end{array}
\]

Suppose that the graph $G$ does not contain any of the graphs listed in the statement of the proposition as induced subgraphs. By Corollary~\ref{cor-split-perm-graph-basis}, it follows that $G$ is a split permutation graph. Therefore $G\cong G_\pi$ for at least one permutation $\pi$ and moreover, every permutation $\pi$ such that $G_\pi\cong G$ is skew-merged. Arbitrarily choose some skew-merged permutation $\pi$ such that $G\cong G_\pi$. We know that if $\beta$ is contained in $\pi$ then $G_\beta$ is an induced subgraph of $G_\pi$. Since $G$ does not contain any of the graphs in the statement of the proposition as induced subgraphs, we see that $\pi$ avoids all of the minimal forbidden permutations listed in Corollary~\ref{cor-WcupWinv-basis}. Therefore $\pi\in\mathcal{W}\cup\mathcal{W}^{-1}$, so $G\in G_\mathcal{W}$, as desired.

For the other direction, consider an arbitrary permutation $\pi\in\mathcal{W}\cup\mathcal{W}^{-1}$. Because $\pi$ is skew-merged, Corollary~\ref{cor-split-perm-graph-basis} implies that $G_\pi$ does not contain $2K_2$, $C_4$, $C_5$, net, co-net, rising sun, or co-rising sun. Moreover, if $G_\pi$ were to contain $H$, $\overline{H}$, cross, co-cross, $X_{168}$, $\overline{X_{168}}$, or $X_{160}$ as induced subgraphs, then such an induced subgraph would be isomorphic to $G_\beta$ for some subpermutation $\beta$ of $\pi$. However, the chart above lists all permutations $\beta$ such that $G_\beta$ is isomorphic to one of these graphs, and every one of these permutations is a forbidden permutation for $\mathcal{W}\cup\mathcal{W}^{-1}$. Therefore we may conclude that every permutation graph in the class $G_\mathcal{W}=G_{\mathcal{W}\cup\mathcal{W}^{-1}}$ avoids all of the induced subgraphs listed in the statement of the proposition, completing the proof.
\end{proof}

\begin{figure}
\begin{footnotesize}
\begin{center}
  \begin{tikzpicture}[scale=0.2, baseline=(current bounding box.center)]
    \draw [thick, darkgray, ->] plot [smooth, tension=0.6] coordinates {(2,-2) (1,2) (-2,1) (-1,-5) (4,-4) (3,4) (-4,3) (-3,-7) (6,-6) (5,6) (-6,5) (-5,-9) (8,-8) (7,8) (-8,7) (-7,-11) (5,-12.5)};
    \plotpartialperm{2/-2, 1/2, -2/1, -1/-5, 4/-4, 3/4, -4/3, -3/-7, 6/-6, 5/6, -6/5, -5/-9, 8/-8, 7/8, -8/7, -7/-11}
	\node at (2,-2) [right] {$v_1$};
	\node at (1,2) [right] {$u_1$};
	\node at (-2,1) [left] {$v_2$};
	\node at (-1,-5) [left] {$u_2$};
	\node at (4,-4) [right] {$v_3$};
	\node at (3,4) [right] {$u_3$};
	\node at (-4,3) [left] {$v_4$};
	\node at (-3,-7) [left] {$u_4$};
	\node at (6,-6) [right] {$v_5$};
	\node at (5,6) [right] {$u_5$};
	\node at (-6,5) [left] {$v_6$};
	\node at (-5,-9) [left] {$u_6$};
	\node at (8,-8) [right] {$v_7$};
	\node at (7,8) [right] {$u_7$};
	\node at (-8,7) [left] {$v_8$};
	\node at (-7,-11) [left] {$u_8$};
  \end{tikzpicture}
\quad\quad\quad\quad
    \begin{tikzpicture}[xscale=1, yscale=2, baseline=(current bounding box.center)]
    \plotpartialperm{1/0,1/1,2/0,2/1,3/0,3/1,4/0,4/1,5/0,5/1,6/0,6/1,7/0,7/1,8/0,8/1};
    \foreach \x in {1,2,3,4,5,6} {
      \draw (\x,0) -- (\x,1);
      \pgfmathsetmacro\ystart{\x + 2};
      \foreach \y in {\ystart,...,8}
      \draw (\x,0) -- (\y,1);
    }
    \draw (8,0) -- (8,1);
    \draw (7,0) -- (7,1);
    \draw (1,1) -- (8,1);
    \node at (1,1) [below right] {$v_1$};
    \node at (2,1) [below right] {$v_2$};
    \node at (3,1) [below right] {$v_3$};
    \node at (8,1) [below right] {$v_8$};
    \node at (1,0) [below right] {$u_1$};
    \node at (2,0) [below right] {$u_2$};
    \node at (3,0) [below right] {$u_3$};
    \node at (8,0) [below right] {$u_8$};
    \foreach \x/\y in {1/3,1/4,1/5,1/6,2/4,2/5,2/6,3/5,3/6,4/6,1/7,2/7,3/7,4/7,5/7,1/8,2/8,3/8,4/8,5/8,6/8}
      \draw (\x,1) to [in=165, out=15] (\y,1);		
  \end{tikzpicture}
\end{center}
\end{footnotesize}
\caption{A widdershins spiral of standard orientation and its permutation graph.}
\label{fig-widdershins-graphs}
\end{figure}
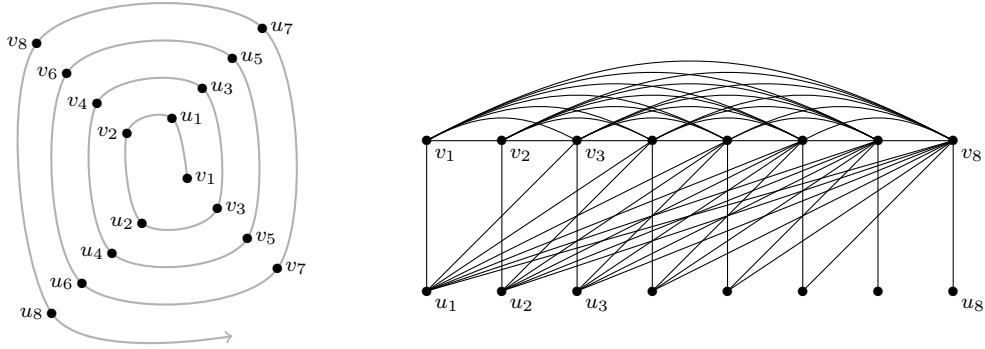

Having established that $G_\mathcal{W}$ is wqo under the induced subgraph relation (as a consequence of Proposition~\ref{prop-W-wqo}, which shows that the permutation class $\mathcal{W}$ is wqo under the permutation containment order) and that $G_\mathcal{W}$ is defined by finitely many minimal forbidden induced subgraphs (the $14$ graphs listed in Proposition~\ref{prop-basis-GW}), our counterexample to Conjecture~\ref{false-conj-lwqo} is completed with our final result, below. The proof we give is adapted from a related argument of Atminas, Brignall, Lozin, and Stacho~\cite{atminas:minimal-classes:}.

\begin{proposition}
\label{prop-GW-not-2wqo}
The class $G_\mathcal{W}$ is not $2$-wqo.
\end{proposition}
\begin{proof}
Let $W_k$ denote the permutation graph of the widdershins spiral of length $2k\ge 8$ and standard orientation. Label the vertices of $W_k$ as $\{u_1,\dots,u_k,v_1,\dots,v_k\}$ as on the left of Figure~\ref{fig-widdershins-graphs}. As shown on the right of this figure, the vertices $\{u_1,\dots,u_k\}$ form an independent set in $W_k$, the vertices $\{v_1,\dots,v_k\}$ form a clique in $W_k$, and we have $u_i\sim v_j$ if and only if $i=j$ or $i\le j-2$. Next we color the vertices $u_1$ and $v_k$ white and the other vertices black, where white and black are a $2$-element antichain, as shown in Figure~\ref{fig-2wqo-antichain}.

\begin{figure}
\begin{center}
	\begin{tikzpicture}[xscale=0.8, yscale=1.2]	
		\plotpartialperm{1/0, 1/1, 2/0, 2/1,3/0, 3/1,4/0, 4/1};
		\foreach \x in {1,2} {
			\draw (\x,0) -- (\x,1);
		  	\pgfmathsetmacro\ystart{\x + 2};
			\foreach \y in {\ystart,...,4}
				\draw (\x,0) -- (\y,1);
		}
		\draw (4,0) -- (4,1);
		\draw (3,0) -- (3,1);
		\draw (1,1) -- (4,1);
		\foreach \x/\y in {1/3,1/4,2/4}
			\draw (\x,1) to [in=165, out=15] (\y,1);		
		\absdothollow{(1,0)};
		\absdothollow{(4,1)};
	\end{tikzpicture}
\quad\quad\quad\quad
	\begin{tikzpicture}[xscale=0.8, yscale=1.2]
		\plotpartialperm{1/0, 1/1, 2/0, 2/1,3/0, 3/1,4/0, 4/1,5/0, 5/1};
		\foreach \x in {1,2,3} {
			\draw (\x,0) -- (\x,1);
		  	\pgfmathsetmacro\ystart{\x + 2};
			\foreach \y in {\ystart,...,5}
				\draw (\x,0) -- (\y,1);
		}
		\draw (4,0) -- (4,1);
		\draw (5,0) -- (5,1);
		\draw (1,1) -- (5,1);
		\foreach \x/\y in {1/3,1/4,1/5,2/4,2/5,3/5}
			\draw (\x,1) to [in=165, out=15] (\y,1);		
		\absdothollow{(1,0)};
		\absdothollow{(5,1)};
	\end{tikzpicture}
\quad\quad\quad\quad
	\begin{tikzpicture}[xscale=0.8, yscale=1.2]
		\plotpartialperm{1/0, 1/1, 2/0, 2/1,3/0, 3/1,4/0, 4/1,5/0, 5/1,6/0,6/1};
		\foreach \x in {1,2,3,4} {
			\draw (\x,0) -- (\x,1);
		  	\pgfmathsetmacro\ystart{\x + 2};
			\foreach \y in {\ystart,...,6}
				\draw (\x,0) -- (\y,1);
		}
		\draw (6,0) -- (6,1);
		\draw (5,0) -- (5,1);
		\draw (1,1) -- (6,1);
		\foreach \x/\y in {1/3,1/4,1/5,1/6,2/4,2/5,2/6,3/5,3/6,4/6}
			\draw (\x,1) to [in=165, out=15] (\y,1);		
		\absdothollow{(1,0)};
		\absdothollow{(6,1)};
	\end{tikzpicture}
\end{center}
\caption{The three elements $W_4$, $W_5$, and $W_6$ of the labelled antichain in the proof of Proposition~\ref{prop-GW-not-2wqo}.}
\label{fig-2wqo-antichain}
\end{figure}
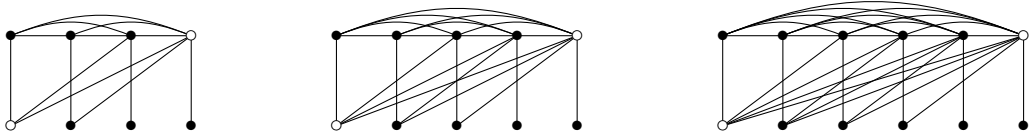
We claim that the set $\{W_k\st k\ge 4\}$ is an infinite antichain of labelled graphs. Suppose to the contrary that some $W_k$ were to embed as a labelled induced subgraph in $W_\ell$ for $\ell>k$. The only copies of $P_4$ in any $W_k$ occur as a set of vertices of the form $\{u_i,v_i,u_{i+1},v_{i+1}\}$. Specifically, there is only one copy of $P_4$ which has a white leaf, namely the graph on the vertices $\{u_1,v_1,u_2,v_2\}$. Similarly, the only copy of $P_4$ in $W_k$ with a white non-leaf is on the vertices $\{u_{k-1},v_{k-1},u_k,v_k\}$. Between these, there is a sequence of $k-3$ black copies of $P_4$, each sharing exactly two vertices with its predecessor copy of $P_4$, and two vertices with its successor copy. Any embedding $W_k$ into $W_\ell$ would have to preserve this sequence of copies of $P_4$, but this is impossible as then we cannot also map both of the copies of $P_4$ with white vertices to their respective positions. This contradiction completes the proof of the proposition.
\end{proof}

\bigskip

\minisec{Acknowledgements}
We thank Jay Pantone for performing the computer search that established Corollary~\ref{cor-WcupWinv-basis}. The computation was performed using the {\tt PermPy} package developed by Homberger and Pantone~\cite{PermPy1.0}.

\bibliographystyle{acm}
\bibliography{../../refs}

\end{document}

%% file: pp-doodles.tex
\newcommand\mybullet{\raisebox{-5pt}{\normalsize \ensuremath{\bullet}}}
\newcommand\mycirc{\raisebox{-5pt}{\normalsize \ensuremath{\circ}}}

\makeatletter
\def\absdot{\@ifnextchar[{\@absdotlabel}{\@absdotnolabel}}
	\def\@absdotlabel[#1]#2{%
		\node at #2 {\normalsize \mybullet};
		\node at #2 [below=2pt] {\ensuremath{#1}};
	}
	\def\@absdotnolabel#1{%
		\node at #1 {\normalsize \mybullet};
	}
\def\absdothollow{\@ifnextchar[{\@absdothollowlabel}{\@absdothollownolabel}}
	\def\@absdothollowlabel[#1]#2{%
		\node at #2 {\normalsize \textcolor{white}{\mybullet}};
		\node at #2 {\normalsize \mycirc};
		\node at #2 [below=2pt] {\ensuremath{#1}};
	}
	\def\@absdothollownolabel#1{%
		\node at #1 {\normalsize \textcolor{white}{\mybullet}};
		\node at #1 {\normalsize \mycirc};
	}
\makeatother

\newcommand{\plotperm}[1]{
	\foreach \j [count=\i] in {#1} {
		\absdot{(\i,\j)};
	};
}

\newcommand{\plotpartialperm}[1]{
	\foreach \i/\j in {#1} {
		\absdot{(\i,\j)};
	};
}

\newcommand{\plotpermbox}[4]{
	\draw [darkgray, thick, line cap=round]
		({#1-0.5}, {#2-0.5}) rectangle ({#3+0.5}, {#4+0.5});
}

\newcommand{\plotpermgraph}[1]{
	\foreach \j [count=\i] in {#1} {
		\foreach \b [count=\a] in {#1} {
			\ifthenelse{\a<\i \AND \b>\j}{\draw (\a,\b)--(\i,\j);}{}
		};
	};
	\plotperm{#1};
}

\newcommand{\plotpermdyckpath}[1]{
	\draw[ultra thick, line cap=round] (0.5,0.5)
	\foreach \step in {#1} {
		\ifnum\step=1
			-- ++(0,1)
		\else
			-- ++(1,0)
		\fi
	};
}

\newcommand{\plotdyckpath}[1]{
	\draw[ultra thick, line cap=round] (0.5,0)
	\foreach \step in {#1} {
		\ifnum\step=1
			-- ++(1,1)
		\else
			-- ++(1,-1)
		\fi
	};
}

\newcommand{\arcskinnyplain}[2]{
	\draw[thick] (#1,0) arc (180:0:{(#2-#1)/2});
}

\newcommand{\matchsmall}[1]{
	\begin{tikzpicture}[scale=.1, anchor=base]
		\def\h{0};
		\def\maxh{0};
		\foreach \i/\j in {#1} {
			\pgfmathparse{\j-\i};
			\let\h\pgfmathresult;
			\pgfmathifthenelse{\h>\maxh}{\h}{\maxh};
			\global\let\maxh\pgfmathresult;
		};
		\pgftransformyscale{{4.5/\maxh}};
		\foreach \i/\j in {#1} {
			\arcskinnyplain{\i}{\j};
		};
	\end{tikzpicture}
}

\newcommand{\matchpermsmall}[1]{
	\begin{tikzpicture}[scale=.1, anchor=base]
		\foreach \j [count=\n] in {#1} {};
		\def\h{0};
		\def\maxh{0};
		\foreach \j [count=\i] in {#1} {
			\pgfmathparse{2*\n+1-\j-\i};
			\let\h\pgfmathresult;
			\pgfmathifthenelse{\h>\maxh}{\h}{\maxh};
			\global\let\maxh\pgfmathresult;
		};
		\pgftransformyscale{{4.5/\maxh}};
		\foreach \j [count=\i] in {#1} {
			\arcskinnyplain{\i}{{2*\n+1-\j}};
		};
	\end{tikzpicture}
}

\newcommand{\plotpinsequence}[1]{
	\absdot{(0,0)}{};
	\edef\n{0}
	\edef\s{0}
	\edef\e{0}
	\edef\w{0}
	\edef\x{0}
	\edef\y{0}
	\foreach \pin [remember=\pin as \oldpin (initially 1), count=\i] in {#1} {
		\ifthenelse{\pin=1 \OR \pin=2}{
			\ifthenelse{\oldpin=3}{
				\xdef\x{\number\numexpr\e-1}
			}{
				\xdef\x{\number\numexpr\w+1}
			}
			\ifnum\i=1 
				\pgfmathparse{\e+1}
 				\xdef\e{\pgfmathresult}
			\fi	
		}{ 
			\ifthenelse{\oldpin=1}{
				\xdef\y{\number\numexpr\n-1}
			}{
				\xdef\y{\number\numexpr\s+1}
			}
			\ifnum\i=1 
				\pgfmathparse{\s-1}
 				\xdef\s{\pgfmathresult}
			\fi	
		}
		\ifnum\pin=1 
			\pgfmathparse{\n+2}
 			\xdef\n{\pgfmathresult}		
			\absdot{(\x,\n)}{};
			\ifnum\i>1
				\draw (\x,\n) -- (\x,\y-0.5);
			\else
			\fi
		\fi
		\ifnum\pin=2 
			\pgfmathparse{\s-2}
 			\xdef\s{\pgfmathresult}
			\absdot{(\x,\s)}{};
			\ifnum\i>1
				\draw (\x,\s) -- (\x,\y+0.5);
			\else
			\fi
		\fi
		\ifnum\pin=3 
			\pgfmathparse{\e+2}
 			\xdef\e{\pgfmathresult}
			\absdot{(\e,\y)}{};
			\ifnum\i>1
				\draw (\e,\y) -- (\x-0.5,\y);
			\else
			\fi
		\fi
		\ifnum\pin=4 
			\pgfmathparse{\w-2}
 			\xdef\w{\pgfmathresult}
			\absdot{(\w,\y)}{};
			\ifnum\i>1
				\draw (\w,\y) -- (\x+0.5,\y);
			\else
			\fi
		\fi		
	};
}

%% file: lwqo.bbl
\def\cprime{$'$}
\begin{thebibliography}{10}

\bibitem{atkinson:restricted-perm:}
{\sc Atkinson, M.~D.}
\newblock Restricted permutations.
\newblock {\em Discrete Math. 195}, 1-3 (1999), 27--38.

\bibitem{atminas:minimal-classes:}
{\sc Atminas, A., Brignall, R. L.~F., Lozin, V., and Stacho, J.}
\newblock Minimal classes of graphs of unbounded clique-width defined by
  finitely many forbidden induced subgraphs.
\newblock arXiv:1503.01628 [math.CO].

\bibitem{atminas:labelled-induce:}
{\sc Atminas, A., and Lozin, V.}
\newblock Labelled induced subgraphs and well-quasi-ordering.
\newblock {\em Order 32}, 3 (2015), 313--328.

\bibitem{avis:on-pop-stacks-i:}
{\sc Avis, D.~M., and Newborn, M.}
\newblock On pop-stacks in series.
\newblock {\em Utilitas Math. 19\/} (1981), 129--140.

\bibitem{benzaken:split-graphs-of:}
{\sc Benzaken, C., Hammer, P.~L., and de~Werra, D.}
\newblock Split graphs of {D}ilworth number {$2$}.
\newblock {\em Discrete Math. 55}, 2 (1985), 123--127.

\bibitem{bose:pattern-matchin:}
{\sc Bose, P., Buss, J.~F., and Lubiw, A.}
\newblock Pattern matching for permutations.
\newblock {\em Inform. Process. Lett. 65}, 5 (1998), 277--283.

\bibitem{brandstadt:graph-classes:-:}
{\sc Brandst{\"a}dt, A., Le, V.~B., and Spinrad, J.~P.}
\newblock {\em Graph Classes: a Survey}.
\newblock SIAM Monographs on Discrete Mathematics and Applications. SIAM,
  Philadelphia, Pennsylvania, 1999.

\bibitem{brignall:decomposing-sim:}
{\sc Brignall, R. L.~F., Huczynska, S., and Vatter, V.~R.}
\newblock Decomposing simple permutations, with enumerative consequences.
\newblock {\em Combinatorica 28}, 4 (2008), 385--400.

\bibitem{daligault:well-quasi-orde:}
{\sc Daligault, J., Rao, M., and Thomass{\'e}, S.}
\newblock Well-quasi-order of relabel functions.
\newblock {\em Order 27}, 3 (2010), 301--315.

\bibitem{damaschke:induced-subgrap:}
{\sc Damaschke, P.}
\newblock Induced subgraphs and well-quasi-ordering.
\newblock {\em J. Graph Theory 14}, 4 (1990), 427--435.

\bibitem{dushnik:partially-order:}
{\sc Dushnik, B., and Miller, E.~W.}
\newblock Partially ordered sets.
\newblock {\em Amer. J. Math. 63\/} (1941), 600--610.

\bibitem{foldes:split-graphs:}
{\sc F{\"o}ldes, S., and Hammer, P.~L.}
\newblock Split graphs.
\newblock {\em Congr. Numer. 14\/} (1977), 311--315.

\bibitem{foldes:split-graphs-ha:}
{\sc F\"oldes, S., and Hammer, P.~L.}
\newblock Split graphs having {D}ilworth number two.
\newblock {\em Canad. J. Math. 29}, 3 (1977), 666--672.

\bibitem{gallai:transitiv-orien:}
{\sc Gallai, T.}
\newblock Transitiv orientierbare {G}raphen.
\newblock {\em Acta Math. Acad. Sci. Hungar. 18\/} (1967), 25--66.

\bibitem{higman:ordering-by-div:}
{\sc Higman, G.}
\newblock Ordering by divisibility in abstract algebras.
\newblock {\em Proc. London Math. Soc. (3) 2\/} (1952), 326--336.

\bibitem{PermPy1.0}
{\sc Homberger, C., and Pantone, J.}
\newblock \textsf{Perm{P}y}.
\newblock Available online at \url{http://permpy.com/}, 2017.

\bibitem{huczynska:well-quasi-orde:}
{\sc Huczynska, S., and Ru{\v{s}}kuc, N.}
\newblock Well quasi-order in combinatorics: embeddings and homomorphisms.
\newblock In {\em Surveys in Combinatorics 2015}, A.~Czumaj, A.~Georgakopoulos,
  D.~Kr{\'a}l', V.~Lozin, and O.~Pikhurko, Eds., vol.~424 of {\em London
  Mathematical Society Lecture Note Series}. Cambridge University Press,
  Cambridge, England, 2015, pp.~261--293.

\bibitem{information-system-on-graph-classes-and-their-inclusions:15:}
{\sc {{I}nformation {S}ystem on {G}raph {C}lasses and their {I}nclusions
  ({ISGCI})}}.
\newblock {P}ublished electronically at \texttt{http://www.graphclasses.org/}.

\bibitem{korpelainen:split-permutati:}
{\sc Korpelainen, N., Lozin, V., and Mayhill, C.}
\newblock Split permutation graphs.
\newblock {\em Graphs Combin. 30}, 3 (2014), 633--646.

\bibitem{korpelainen:boundary-proper:}
{\sc Korpelainen, N., Lozin, V., and Razgon, I.}
\newblock Boundary properties of well-quasi-ordered sets of graphs.
\newblock {\em Order 30}, 3 (2013), 723--735.

\bibitem{murphy:restricted-perm:}
{\sc Murphy, M.~M.}
\newblock {\em Restricted Permutations, Antichains, Atomic Classes, and Stack
  Sorting}.
\newblock PhD thesis, University of St Andrews, 2002.
\newblock Available online at {\texttt{http://hdl.handle.net/10023/11023}}.

\bibitem{pouzet:un-bel-ordre-da:}
{\sc Pouzet, M.}
\newblock Un bel ordre d'abritement et ses rapports avec les bornes d'une
  multirelation.
\newblock {\em C. R. Acad. Sci. Paris S\'er. A-B 274\/} (1972), A1677--A1680.

\bibitem{robertson:graph-minors-i-xx:}
{\sc Robertson, N., and Seymour, P.}
\newblock Graph minors {I}--{XX}.
\newblock {\em J. Combin. Theory Ser. B\/} (1983--2004).

\bibitem{stankova:forbidden-subse:}
{\sc Stankova, Z.~E.}
\newblock Forbidden subsequences.
\newblock {\em Discrete Math. 132}, 1-3 (1994), 291--316.

\bibitem{vatter:permutation-cla:}
{\sc Vatter, V.~R.}
\newblock Permutation classes.
\newblock In {\em Handbook of Enumerative Combinatorics}, M.~B{\'o}na, Ed. CRC
  Press, Boca Raton, Florida, 2015, pp.~754--833.

\end{thebibliography}
